\documentclass[a4paper,10pt]{article}

\usepackage[francais,english]{babel}
\usepackage[latin1]{inputenc}   
\usepackage{amsmath,amsthm}
\usepackage[dvips]{graphicx}
\usepackage{amsfonts,amssymb}
\usepackage{geometry}
\usepackage{hyperref}
\usepackage{stmaryrd}
\usepackage{fancyhdr}
\usepackage{url}
\usepackage{dsfont}

\newtheorem*{conj*}{Conjecture}

\newtheorem*{ack}{Acknowledgements}
\newtheorem*{thm*}{Theorem}

\newtheorem{prop}{Proposition}[section]

\newtheorem{LM}{Lemma}[section]

\newtheorem{thm}{Theorem}[section]
\newtheorem{df}{Definition}[section]
\newtheorem{cor}{Corollary}[section]

\newtheoremstyle{pourlesremarques}{\topsep}{\topsep}{\normalfont}{}{\bfseries}{.}{ }{}
\theoremstyle{pourlesremarques}

\newtheorem*{rem*}{Remark}
\newtheoremstyle{pourlesexemples}{\topsep}{\topsep}{\normalfont}{}{\bfseries}{.}{ }{}
\theoremstyle{pourlesexemples}

\renewcommand{\o}{\mathfrak{O}}

\newcommand{\w}{\varpi}
\renewcommand{\d}{\delta}
\newcommand{\R}{\mathbb{R}}
\renewcommand{\l}{\lambda}
\newcommand{\C}{\mathbb{C}}
\newcommand{\Q}{\mathbb{Q}}

\newcommand{\M}{\mathcal{M}}
\newcommand{\Z}{\mathbb{Z}}

\newcommand{\0}{\mathbf{0}}

\newcommand{\D}{\Delta}

\title {\textbf{Shalika periods and parabolic induction for GL(n) over a non archimedean local field}}
\author{Nadir MATRINGE\footnote{Nadir Matringe, Universit\'e de Poitiers, Laboratoire de Math\'ematiques et Applications,
T\'el\'eport 2 - BP 30179, Boulevard Marie et Pierre Curie, 86962, Futuroscope Chasseneuil Cedex. Email: Nadir.Matringe@math.univ-poitiers.fr}}

\begin{document}
\maketitle

\begin{abstract}
Let $F$ be a non archimedean local field, and $n_1$ and $n_2$ two positive even integers. We prove that if $\pi_1$ and $\pi_2$ are two smooth 
representations of $GL(n_1,F)$ and $GL(n_2,F)$ respectively, both admitting a Shalika functional, then the 
normalised parabolically induced representation $\pi_1\times \pi_2$ also admits a Shalika functional. Combining this with 
the results of \cite{M-localBF}, we obtain as a corollary the classification of generic representations of $GL(n,F)$ admitting a Shalika functional 
when $F$ has characteristic zero. This result is relevant to the study of the Jacquet-Shalika exterior square $L$ factor.
\end{abstract}

\section{Introduction}

Let $F$ be a non archimedean local field, let $n$ be an even positive integer, and $\pi$ be a smooth irreducible representation of 
$GL(n,F)$ on a complex vector space. Jacquet and Shalika 
introduced in \cite{JS} an integral representation of the exterior square $L$ factor of $\pi$ (though the work \cite{JS} is more concerned 
with the global exterior square $L$ function). A more detailed study of the exterior square $L$ factor of $\pi$ was started by Cogdell and 
Piatetski-Shapiro (see \cite{CP}), and recently continued in \cite{K}, \cite{B12}, \cite{MY13}, \cite{KR}, \cite{M-linearmirabolic} and \cite{CM}. 
The existence of a Shalika functional on $\pi$ (see Definition \ref{shalika-period}) 
is related with the occurence of a certain type of pole at $0$ of the exterior square $L$ factor of $\pi$ 
(see \cite{K} when $\pi$ is square integrable). In particular, according to the Langlands functoriality conjectures, it is expected that $\pi$ will have a Shalika functional precisely when it is 
a functorial lift from $SO(2n+1,F)$ (see for example \cite{JNQ} for the cuspidal case). Another way to say this is that if $\pi$ has Langlands parameter $\phi$ (which is a semi-simple
representation of the Weil-Deligne group $W'_F$ of $F$ on a complex vector-space $V$ of dimension $n$), then $\pi$ should admit a Shalika functional if and only if the image of $\phi$ fixes a non-degenerate alternating form $B$ on $V \times V$. But then, 
if $\phi_1$ and $\phi_2$ are two finite dimensional semi-simple
representations of $W'_F$, the image of which respectively fix two non-degenerate alternating forms $B_1$ and $B_2$, then the image of $\phi_1\oplus \phi_2$ fixes the non-degenerate alternating form $B_1\oplus B_2$. As the
direct sum of Langlands parameters corresponds to parabolic induction, at least in the case where parabolic induction preserves irreducibility, it is natural to expect that the property
"to admit a Shalika functional" should be stable under parabolic induction (notice that a similar property has been established in \cite{Kap16} 
in the context of the symmetric square $L$ factor). The main goal of this note 
is to prove this fact.

\begin{thm}\label{main}
Let $n_1$ and $n_2$ be two even integers, and $\pi_1$ and $\pi_2$ be two smooth 
representations of $GL(n_1,F)$ and $GL(n_2,F)$ respectively, both admitting a Shalika functional, then the 
normalised parabolically induced representation $\pi_1\times \pi_2$ of $GL(n_1+n_2,F)$ also admits a Shalika functional.
\end{thm}

As a corollary, when $F$ is of characteristic zero, using the results of \cite{M-localBF}, we obtain the following consequence 
(the ``symmetric square'' analogue being the main result of 
\cite{Kap15}), which is stated in terms of the Zelevinsky classification 
of generic representations (see \cite{Z}, Theorem 9.7).

\begin{cor}\label{cor}
Let $n$ be a positive even integer, and $\pi$ be a generic representation of $GL(n, F )$.
The representation $\pi$ admits a Shalika functional if and only if it is obtained as a normalised
parabolic induction
\[(\D_1\times \D_1^\vee)\times \dots \times (\D_s\times \D_s^\vee)\times \D_{s+1}\times \dots \times \D_t , \ (0 \leq s \leq  t)\]
where each $\D_i$ is a discrete series, which moreover admits a Shalika functional for $i>s$. In
particular $\pi$ admits a linear period if and only if it admits a Shalika functional.                                                                                                                          
\end{cor}

This corollary is also a consequence of the results, obtained by a different method, of the forthcoming work \cite{G}, 
which relates linear and Shalika functionals via the Theta correspondance for the pair $(GL(n,F),GL(n,F))$. It is quite 
clear, in light of 
the works \cite{M-asai} and \cite{M-localBF}, that this result should play a role in establishing the inductivity 
relation of 
the exterior square $L$ factor.

{\ack{ I thank the referee for his careful reading and his many helpful suggestions, the paper in his actual form owes him much. This work was supported by the research project ANR-13-BS01-0012 FERPLAY.}}

\section{Preliminaries}

We let $F$ be a non archimedean local field, $\o$ be its ring of integers, $\w$ a uniformiser, $q$ the residual cardinality 
$|\o/\w \o|$, and 
$\mid.\mid$ the absolute value on $F$, normalised by the condition $|\w|=q^{-1}$. 
We set $\M_{p,q}$ the matrix algebra $\M(p,q,F)$, we will denote $\M_{p,p}$ by $\M_p$. We will write 
$\0_{p,q}$ (or $\0_p$ if $p=q$) the zero 
matrix of $\M_{p,q}$. We denote by 
$\M_{p,q}^o$ the Zariski open subset of matrices of maximal rank $\mathrm{min}(p,q)$ in $\M_{p,q}$, it is the $F$-points 
of a Zariski open set.  We will denote 
by $G_p$, the general linear group of invertible elements in $\M_p$. We will denote by $K_p$ the group $G_p(\o)$, 
and for $l\geq 1$, by $K_p(l)$ its congruence subgroup $I_p+\w^l\M_p(\o)$.\\

Let $n_1=2m_1$ and $n_2=2m_2$ be two positive \textit{even} integers, with $n_1\geq n_2$, 
let $n=n_1+n_2$, and $m=m_1+m_2=n/2$.
We denote by $P=MU$ the standard (i.e. containing the Borel subgroup of upper triangular matrices) 
parabolic subgroup of $G=G_n$ of type $(n_1,n_2)$, with unipotent radical $U$, and Levi subgroup $M$ given by 
the appropriate block diagonal matrices. 
Let $\pi_i$ be an admissible representation of $G_{n_i}$. We denote by $G'$ the subgroup of $G$ 
of matrices $diag(g,g)$ with $g$ in $G_m$, and by $P'$ the subgroup of $G'$ 
of matrices $diag(p,p)$ with $p$ in the standard parabolic subroup of $G_m$ of type $(m_1,m_2)$ containing 
the Borel subgroup 
of upper triangular matrices. We denote by $N$ the unipotent radical of the standard parabolic subroup 
of $G_m$ of type $(m,m)$, 
so that 
$S=NG'$ is what is known as the Shalika subgroup of $G$. We fix a nontrivial character $\theta$ of $(F,+)$: 
this defines 
a character $\Theta$ of $S$, given by the formula $\Theta(ng')= \theta(Tr(x))$, where 
$n=\begin{pmatrix} I_m & x \\ & I_m \end{pmatrix}$.

\begin{df}\label{shalika-period}
If $\pi$ is a smooth representation of $G$, one says that $\pi$ has a Shalika functional if 
$\mathrm{Hom}_S(\pi,\Theta)\neq 0$, in which case we call a Shalika functional on $\pi$ an element
of $\mathrm{Hom}_S(\pi,\Theta)-\{0\}$.
\end{df}

Let $w$ be the Weyl element of $G$, which is the permutation matrix corresponding to the permutation:
\begin{itemize}
 \item $i\mapsto i$ if $1\leq i\leq m_1$.
\item $i\mapsto i+m_2$ if $m_1<i\leq 2m_1$.
\item $i\mapsto i-m_1$ if $2m_1<i\leq 2m_1+m_2$.
\item $i\mapsto i$ if $2m_1+m_2<i\leq n=2m_1+2m_2$.
\end{itemize}

We denote by $Q$ the group $P^w=w^{-1}Pw$, by $L$ the group $M^w$, by $V$ the group $U^w$, and by $V^-$ the 
image of $V$ 
under transpose. We denote by $V_0^-$ the group $V^-\cap N$. The elements of $V^-$ are of the form 
$$\begin{pmatrix} I_{m_1} &  &  & \\ a & I_{m_2} & b & \\ &  & I_{m_1}& \\ c & & d & I_{m_2} \end{pmatrix},$$ with 
$a, b, c, d\in \M_{m_2,m_1}$, and those of $V_0^-$ are of the form 
$$v_0^-(x)=\begin{pmatrix} I_{m_1} &  &  & \\ & I_{m_2} & x & \\ &  & I_{m_1}& \\ & & & I_{m_2} \end{pmatrix},$$ with 
$x\in \M_{m_2,m_1}$. 
The elements of $Q\cap N$ are those of the form 
$$\begin{pmatrix} I_{m_1} &  & u & v \\ & I_{m_2} &  & w\\ &  & I_{m_1}& \\ & & & I_{m_2} \end{pmatrix},$$ for
$u\in \M_{m_1}$, $v \in \M_{m_1,m_2}$, and $w\in \M_{m_2}$. In particular one has 
$N=V_0^- (Q\cap N)$, and the
elements in $V_0^-$ commute with those in $\Q\cap N$. Let $\pi_i$ be 
a smooth representation of $G_{n_i}$, with a nonzero Shalika functional $\l_i$ on its space. 
Let $\sigma$ be the inflation of 
$\pi_1\otimes \pi_2$ to $P$, and 
\[\l=\l_1\otimes \l_2:v_1\otimes v_2 \mapsto \l_1(v_1)\l_2(v_2)\]
the associated linear form on the space of $\sigma$. 
Let $\tau=\sigma^w$ the corresponding representation of $Q$, acting on the same space as $\sigma$. 
We denote by $I$ the representation $\mathrm{Ind}_P^G(\sigma)$, and by $J$ the representation $\mathrm{Ind}_Q^G(\tau)$, 
where the parabolic induction is normalised. Clearly $J$ and $I$ are isomorphic $G$-modules.

\section{Proof of the Theorem \ref{main}}

 For $f$ in $J$, 
we denote by $\l(f)$ the map $\l\circ f$. Let $\nu=|.|\circ det$, and $\d_{P'}$ be the character 
of $P'$ defined by 
$$\d_{P'}(p')=\nu(a)^{m_2}\nu(c)^{-m_1}$$ for 
$$p'= diag \left(\begin{pmatrix} a & b \\ & c \end{pmatrix},\begin{pmatrix} a & b \\ & c \end{pmatrix}\right).$$
Note that if we denote by $\d_P$ the modulus character of $P$, then 
$\d_{P'}(p')=\d_P^{1/2}(p')$ for all $p'$ in $P'$.
We observe that $\l(f)$ satisfies the three following relations:
\begin{LM}
for $v_0^-$ in $V_0^-$, $n$ in $Q\cap N$, $l'$ in 
$L'=L\cap G'$, and $v'$ in $V'=V\cap G'$, the following equalities hold.
\begin{equation}\label{1}\l(\rho(n)f)(v_0^-)=\l(f(v_0^-n))=\l(f(nv_0^-))=\l(\tau(n)f(v_0^-))=\Theta(n)\l(f)(v_0^-)\end{equation} 
\begin{equation}\label{2}\l(\rho(l')f)(v_0^-)=\l(f(v_0^- l'))=\l(f(l'(v_0^-)^{l'}))=\d_{P'}(l')\l(f)((v_0^-)^{l'})\end{equation}
\begin{equation}\label{3}\l(\rho(v')f)(v_0^-)=\l(f)(v_0^-)\end{equation} 
\end{LM}
\begin{proof} Call $N_1$ and $N_2$ the unipotent radicals of the Shalika subgroups of 
$G_{n_1}$ and $G_{n_2}$ respectively, then $Q\cap N$ is a subgroup of $diag(N_1,N_2)^w.V$, hence for any $n\in Q\cap N$, 
and any $v$ in $\pi_1\otimes \pi_2$, one has $\l(\tau(n)v)=\Theta(n)\l(v)$, and this gives Relation (\ref{1}). 
Relation $(\ref{2})$ comes from the $\tau(L')$-invariance of $\l$, and the fact that we used normalised 
parabolic induction to define $J$. For Relation (\ref{3}), we write 
$$v_0^-=\begin{pmatrix} I_{m_1} & & & \\ &I_{m_2}& x & \\ & &I_{m_1}& \\ & & &I_{m_2} \end{pmatrix},\ v'=\begin{pmatrix} I_{m_1} & y& & \\ &I_{m_2}&  & \\ & &I_{m_1}& y\\ & & &I_{m_2} \end{pmatrix},$$ then 
$${v'}^{-1}v_0v'=\begin{pmatrix} I_{m_1} & &-yx & -yxy\\ &I_{m_2}& x & xy\\ & &I_{m_1}& \\ & & &I_{m_2} \end{pmatrix}=nv_0,$$ where 
the matrix $n$ is equal to $\begin{pmatrix} I_{m_1} & &-yx & -yxy\\ &I_{m_2}&   & xy\\ & &I_{m_1}& \\ & & &I_{m_2} \end{pmatrix}$. 
We notice that $n$ 
belongs to $Q\cap N$, and $$\Theta(n)=\theta(Tr(-yx+xy))=1,$$ so we deduce that 
$$\l(f(v_0^-v'))=\l(f(v'{v'}^{-1}v_0^-v'))=\l(f({v'}^{-1}v_0^-v'))= 
\l(f(nv_0^{-}))=\Theta(n)\l(f(v_0^{-}))=\l(f(v_0^{-})).$$
\end{proof}

The following lemma will be needed. As its proof is the most technical part of the paper, we postpone it to the 
next section, and admit it for the moment. 

\begin{LM}\label{key}
There is a compact subset $C$ of $\M_{m_2,m_1}$, such that for any $f\in J$, the map 
$$x\mapsto  \l(f)(v_0^-(x))$$ restricted to $\M_{m_2,m_1}^o$
 vanishes outside $\M_{m_2,m_1}^o\cap C$. In particular, 
the integral $$\phi(f)=\int_{x\in\M_{m_2,m_1}} \l(f)(v_0^-(x))dx$$ converges absolutely.
\end{LM}

We explain right now why the assertion on the support in the lemma above implies the absolute convergence: as $\M_{m_2,m_1}\backslash \M_{m_2,m_1}^o$ is the set of $F$-points of a Zariski closed subset of $\M_{m_2,m_1}$, it has measure zero, and thus   
$$\phi(f)=\int_{x\in\M_{m_2,m_1}^o} \l(f)(v_0^-(x))dx=\int_{x\in \M_{m_2,m_1}^o \cap C} \l(f)(v_0^-(x))dx = 
\int_{x\in C} \l(f)(v_0^-(x))dx,$$ hence the convergence by smoothness of 
the map $x\mapsto \l(f)(v_0^-(x))$ on $\M_{m_2,m_1}$.\\ 

Thanks to Relation (\ref{1}) above, the linear form $\phi$ on $J$ is $\Theta$-equivariant under the groups $Q\cap N$, but it is also $\Theta$-equivariant under $V_0^-$ by definition (notice 
that $\Theta_{|V_0^-}$ is trivial), hence it is $\Theta$-equivariant under $N=V_0^-.Q\cap N$. 
We set $\chi=\d_{P'}\d^{1/2}=(\d_{P}^{1/2})_{|P'}$. We recall that the
space $\mathcal{C}_c^\infty(P'\backslash G',\chi)$ of smooth functions on $G'$ which staisfy (i) $f(hg) = \chi(h)f(g)$ for $h \in P'$ and
$g \in G'$ (and (ii) the support of $f$ is compact mod the left action of $P'$, which is automatic as $P'\backslash G'$ is compact), possesses a non-zero right $G'$-invariant linear form $d\mu_{P'\backslash G'}$. Relations (\ref{2}) and (\ref{3}) imply that $\phi$ is $\chi$-equivariant under $P'$. In particular, because of this equivariance under 
$P'$, because $P'\backslash G'$ is compact, and because $g\mapsto \phi(\rho(g)f)$ is smooth, the linear form $$\Phi:f\mapsto \int_{P'\backslash G'} \phi(\rho(g)f)d\mu_{P'\backslash G'}(g)$$ is well defined 
on $J$. It is $G'$-invariant by definition, and for $n\in N$, one has 
$$\Phi(\rho(n)f)=\int_{P'\backslash G'} \phi(\rho(gn)f)d\mu_{P'\backslash G'}(g)= \int_{P'\backslash G'} \Theta(n^g)\phi(\rho(g)f)d\mu_{P'\backslash G'}(g)=\Theta(n)\Phi(f)$$
because $\Theta(n^g)=\Theta(n)$ for any $g\in G'$. In particular 
$$\Phi\in \mathrm{Hom}_S(J,\Theta),$$ hence we will be done if we show that 
$\Phi$ is nonzero.\\

 We first remark that if 
$$U'^-=\{diag\left(\begin{pmatrix} I_{m_1} & \\ y & I_{m_2}\end{pmatrix} , 
\begin{pmatrix} I_{m_1} & \\ y & I_{m_2}\end{pmatrix}\right), \ y\in \M_{m_2,m_1}\}$$ is the opposite of the unipotent radical of $P'$, and $h$ is a map 
in the subspace of functions $J'_0$ of $J'=\mathcal{C}_c^\infty(P'\backslash G',\chi)$, the restriction of which to $U'^-$ has compact support, then 
$$\int_{P'\backslash G'} h(g)d\mu_{P'\backslash G'}(g)= \int_{U'^-} h(u)du.$$
Let $J_0$ be the subspace of functions in $J$, which restrict to $V^-$ with compact 
support, it is clear that if $f\in J_0$, then $g\in G'\mapsto \phi(\rho(g)f)$ belongs to $J'_0$, and one has 
$$\Phi(f)=\int_{x,y \in \M_{m_2,m_1}} \l(f)\left( \begin{pmatrix} I_{m_1} &  &  & \\ y & I_{m_2} & x & \\ &  & I_{m_1}& \\  & & y & I_{m_2} \end{pmatrix} \right)dx dy.$$
Now let $v_i$ be a vector in the space of $\pi_i$ with $\l_i(v_i)\neq 0$, both fixed under $K_n(r)$ for some $r>0$ large enough contained in the $Q$-parahoric subgroup $K_Q$ of $K_n$. 
Then for $v=v_1\otimes v_2$, there is a unique map $f\in J_0$, with support equal to 
$QK_n(r)$, such that $f(qk)=\tau(q)v$, for $q\in Q$ and 
$k \in K_n(r)$. For such an $f$, it follows that the integrand in $\Phi(f)$ has compact open support $K_f$ in $\mathcal{M}_{m_2 ,m_1}$, and $\Phi(f) = vol(K_f )\lambda(v)$, hence is nonzero. This ends the proof of Theorem \ref{main}.

\section{Proof of Lemma \ref{key}}

We will need the following property of Shalika functions (see Lemma 3.1. of \cite{FJ} or Lemma 4.1 of \cite{M-linearmirabolic} 
for the proof in the non archimedean case).

\begin{prop}\label{shalika-vanish}
Let $\rho$ be a representation of $G_{2m}$, and $L$ a Shalika functional on $\rho$, then for any 
$v$ in $V$, the map $g\mapsto L(\rho(diag(g,1))v)$ is supported on the intersection of a compact set of 
$\M_m$ with $G_m$.
\end{prop}

In the proof of Lemma \ref{main}, we will use the representation $I$ rather than $J$, hence the group of matrices of the form 
$$\begin{pmatrix} I_{m_1} & & & \\ & I_{m_1} & & \\  & x & I_{m_2} & \\  &  & & I_{m_2} \end{pmatrix}$$ with $x\in 
\M_{m_2,m_1}$ instead of 
$V_0^-$. We now explain how to prove Lemma \ref{main} when $n_1=n_2=2$. This particular case contains the main steps of the general computation, 
but is easy to write. We use the familiar matrix relation valid for $x\neq 0$:
$$\begin{pmatrix} 1 & 0 \\ x & 1 \end{pmatrix}=
\begin{pmatrix} -x^{-1} & 1 \\ 0 & x \end{pmatrix}\begin{pmatrix} 0 & 1 \\ 1 & 0 \end{pmatrix}\begin{pmatrix} 1 & x^{-1} \\ 0 & 1 \end{pmatrix}.$$

Let $f$ belong to $I$, we choose $l$ large enough, such that $f$ is right invariant under $K_4(l)$. One has for $|x|\geq q^l$: 
$$\l(f)\left(\begin{pmatrix} 1 & & & \\ & 1 &  & \\ & x & 1 & \\ & & & 1 \end{pmatrix}\right)=
\l(f)\left(\begin{pmatrix} 1 & & & \\ & -x^{-1} & 1 & \\ &  & x & \\ & & & 1 \end{pmatrix}
\begin{pmatrix} 1 & & & \\ &  & 1 & \\ & 1 &  & \\ &  & & 1 \end{pmatrix}
\begin{pmatrix} 1 & & & \\ & 1 & x^{-1} & \\ &  & 1 & \\ & & & 1 \end{pmatrix}\right)$$
$$=|x|^{-2}\l(\pi_1(\mathrm{diag}(-x,1))\otimes \pi_2 (\mathrm{diag}(x,1))(f\begin{pmatrix} 1 & & & \\ &  & 1 & \\ & 1 &  & \\ &  & & 1 \end{pmatrix})).$$ 

Writing $f\begin{pmatrix} 1 & & & \\ &  & 1 & \\ & 1 &  & \\ &  & & 1 \end{pmatrix}$ as a sum 
$\sum_i v_1^i\otimes v_2^i,$ we obtain 
$$\l(f)\left(\begin{pmatrix} 1 & & & \\ & 1 &  & \\ & x & 1 & \\ & & & 1 \end{pmatrix}\right)
=|x|^{-2}\sum_i \l_1(\mathrm{diag}(-x,1)v_1^i) \l_2(\mathrm{diag}(x,1)v_2^i),$$ so there is $k\geq l$, such that this quantity 
is zero for $|x|> k$ according to Lemma \ref{shalika-vanish}. Hence the support of the map $x\mapsto \l(f)(v_0^-(x))$ is contained in $\{x\in F,|x|\leq q^k\}$, and Lemma \ref{main} is proved in this special case.\\

We now give the proof in general. Let us set $\Lambda^+(m)=\{a\in \Z^m, a_1\leq \dots \leq a_m\}$. For 
$(x_1,\dots,x_{m_2})\in F^{m_2}$, we write 
$$u(x_1,\dots,x_{m_2})=\begin{pmatrix}\mathrm{diag}(x_1,\dots,x_{m_2})| \mathbf{0}_{m_2,m_1-m_2} \end{pmatrix}\in \mathcal{M}_{m_2,m_1}$$ 
and 
$$v(x_1,\dots,x_{m_2})=\begin{pmatrix}\mathrm{diag}(x_1,\dots,x_{m_2})\\ \mathbf{0}_{m_1-m_2,m_2} \end{pmatrix}\in \mathcal{M}_{m_1,m_2}$$

We first recall a straightforward consequence of the matricial form of the structure theorem of finitely generated 
modules over a principal ideal domain.

\begin{prop}
One has: $$\mathcal{M}_{m_2,m_1}^o=\coprod_{a\in \Lambda^+(m_2)} K_{m_2}
 u(\w^{a_1},\dots,\w^{a_{m_2}}) K_{m_1}.$$
\end{prop}

We continue with the proof of Lemma \ref{key}. To lighten notations, we set $f^\l$ the map on $\M_{m_2,m_1}$ defined 
by $$f^\l(x)=\l(f)\begin{pmatrix} I_{m_1} & & & \\ & I_{m_1} & & \\ & x & I_{m_2}& \\ & & & I_{m_2} \end{pmatrix}.$$ For any $b\in \Z$, the set $$C(b)=\coprod_{b\leq a_1 \leq \dots \leq a_{m_1}} K_{m_2}
u(\w^{a_1},\dots,\w^{a_{m_2}}) K_{m_1}$$ is relatively compact 
in $\M_{m_2,m_1}$, for example it is contained in the compact set $\w^b\M_{m_2,m_1}(\o)$. We will show that $f^\l$ restricted to 
$\M_{m_2,m_1}^o$ vanishes outside $C(b)$ for some $b\in \Z$. 
First, we notice that 
$$f^\l(k_2xk_1)=  \l(f)\left( \begin{pmatrix} I_{m_1} & & & \\ & I_{m_1} & & \\ &  & k_2& \\ & & & I_{m_2} \end{pmatrix}
\begin{pmatrix} I_{m_1} & & & \\ & I_{m_1} & & \\ & x & I_{m_2}& \\ & & & I_{m_2} \end{pmatrix}
\begin{pmatrix} I_{m_1} & & & \\ & k_1 & & \\ &  & I_{m_2}& \\ & & & I_{m_2} \end{pmatrix}\right),$$ 
is equal to $$\l(f)\left( \begin{pmatrix} I_{m_1} & & & \\ & I_{m_1} & & \\ &  & I_{m_2} & \\ & & & k_2^{-1} \end{pmatrix}
\begin{pmatrix} I_{m_1} & & & \\ & I_{m_1} & & \\ & x & I_{m_2}& \\ & & & I_{m_2} \end{pmatrix}
\begin{pmatrix} I_{m_1} & & & \\ & k_1 & & \\ &  & I_{m_1}& \\ & & & I_{m_2} \end{pmatrix}\right)$$ because of left equivariance 
property of $f$, and the invariance of $\l$ under $\mathrm{diag}(g,g,h,h)$ for $(g,h)\in G_{m_1}\times G_{m_2}$. So we obtain  
$$f^\l(k_2xk_1)=(\rho(k)f)^\l(x),$$ for $k=\mathrm{diag}(I_{m_1}, k_1, I_{m_2},k_2^{-1})$. As $f$ is fixed by some open subgroup of $G_n$ 
on the right, there are only a finite number of maps $\rho(k)f$ when $k$ varies in $K_n$, hence 
it suffices to prove that for any $f$ in $I$, there is $b\in \Z$, such that the
 map $$(a_1,\dots,a_{m_2})\mapsto f^\l (u(\w^{a_1},\dots,\w^{a_{m_2}}))$$ from the set 
$\Lambda^+(m_2)$ to $\C$, vanishes whenever $a_1<b$. We will in fact prove by induction that for $r$ between 
$1$ and $m_2$, there is $b_r$ such that the map 
$$(a_1,\dots,a_r)\mapsto f^\l (u(\w^{a_1},\dots,\w^{a_r},0,\dots,0))$$ from $\Lambda^+(r)$ to $\C$, vanishes whenever $a_1<b_r$.\\

We only do the induction step, as the case $r=1$ is up to notational changes the particular case we treated 
before the proof. Take $l$ such that $f$ is fixed under right translation by $K_n(l)$. In particular if $a_r\leq -l$, the matrix 
$$k'=\begin{pmatrix} I_{m_1} & & & \\ & I_{m_1} & v(\w^{-a_1},\dots,\w^{-a_r},0,\dots,0) & \\ &  & I_{m_2} & \\ & & & I_{m_2} \end{pmatrix}$$fixes $f$ on the right, and denoting by $W_r$ the matrix 
$$W_r=\begin{pmatrix} I_{m_1} & & & \\ & \mathrm{diag}(\0_{r},I_{m_1-r}) & v(1,\dots,1,0,\dots,0) & \\ & u(1,\dots,1,0,\dots,0) & \mathrm{diag}(\0_{r},I_{m_2-r}) & \\ & & & I_{m_2} \end{pmatrix},$$ one has 

$$f^\l (u(\w^{a_1},\dots,\w^{a_r},0,\dots,0))=$$
$$\l(f)\left( \begin{pmatrix} I_{m_1} & & & \\ 
& \mathrm{diag}(-\w^{-a_1},\dots,-\w^{-a_r},1,\dots,1)& v(1,\dots,1,0,\dots,0) & \\
 &  & \mathrm{diag}(\w^{a_1},\dots,\w^{a_r},1,\dots,1) & \\ & & & I_{m_2} \end{pmatrix}W_r k' \right)$$
$$=\l(f)\left( \begin{pmatrix} \mathrm{diag}(\w^{a_1},\dots,\w^{a_r},1,\dots,1) & & & \\ & I_{m_1} & v(1,\dots,1,0,\dots,0)& \\ &  & 
\mathrm{diag}(\w^{a_1},\dots,\w^{a_r},1,\dots,1) & \\ & & & I_{m_2} \end{pmatrix}W_r\right).$$

Write $f(W_r)=\sum_i v_1^i\otimes v_2^i$, then $$q^{m(a_1+\dots+a_r)}f^\l (u(\w^{a_1},\dots,\w^{a_r},0,\dots,0))$$
$$=\sum_i \l_1(\pi_1(\mathrm{diag}(\w^{a_1},\dots,\w^{a_r},1,\dots,1, I_{m_1}))v_1^i)
\l_2(\pi_1(\mathrm{diag}(\w^{a_1},\dots,\w^{a_r},1,\dots,1, I_{m_2}))v_2^i).$$

For $y$ in $F$, we introduce the notation 
\[f_y=\rho\begin{pmatrix}\begin{pmatrix} I_{m_1} & & & \\   & I_{m_1} & & \\ &u(\underbrace{0,\dots,0}_{r-1},y,0,\dots,0) & I_{m_2} & \\ & & & I_{m_2} \end{pmatrix}W_r\end{pmatrix} f,\]
hence we have the relation
\[f^\lambda(u(x_1 ,\dots,x_r ,0,\dots,0)) = f^\lambda_{x_r} (x_1 ,\dots,x_{r-1} ,0,\dots,0)).\]

Now by smoothness of $f$, there is a finite number $x_{r,1},\dots,x_{r,d}$ of elements in 
$\w^{k}\o$, for any $(x_1,\dots,x_r)\in F^{r-1}\times \w^{k}\o$, then $$f^\l(u(x_1,\dots,x_r,0,\dots,0))= 
f^\lambda_{x_{r,i}} (x_1 ,\dots,x_{r-1} ,0,\dots,0))$$ for some $i$ between $1$ and $d$. The induction 
hypothesis gives integers $b_{r-1,i}$ such that $$f_{x_{r,i}}^\l(u(\w_1^{a_1},\dots,\w_{r-1}^{a_{r-1}},0,0,\dots,0))$$ 
vanishes for any $(a_1,\dots,a_{r-1})\in \Lambda^+(r-1)$ such that $a_1<b_{r-1,i}$. It suffices to take 
$$b_r=min(b_{r-1,1},\dots,b_{r-1,d}),$$ to obtain that $f^\l(u(\w^{a_1},\dots,\w^{a_r},0,\dots,0))$ vanishes whenever $a_1<b_r$.

\section{Generic representations with a Shalika model}

This last section is devoted to the proof of Corollary \ref{cor}. We will use linear periods.

\begin{df}
Let $n=2m$ be a positive even integer, and let $H_n$ be the Levi subgroup of $G_n$ given by matrices 
$\mathrm{diag}(g_1,g_2)$, with $g_i\in G_m$. One says that a smooth representation $\pi$ admits a linear period 
if the space $\mathrm{Hom}_{H_n}(\pi,\C)$ is nonzero. We call a linear form in $\mathrm{Hom}_{H_n}(\pi,\C)-\{0\}$ a linear period.
\end{df}

We recall Proposition 3.1 of \cite{FJ}, which applies to any irreducible (and in fact finite length) representation of $G_n$ 
thanks to Lemma 6.1 of \cite{JR}, as explained in the remark after the aforementioned lemma in [ibid.].

\begin{prop}\label{shalika-implies-linear}
Let $n$ be a positive even integer, and $\pi$ an irreducible smooth representation of $G_n$ with a Shalika functional, then it has a linear period. 
\end{prop}

We recall a consequence of Theorem 5.1. of \cite{M-linearmirabolic}, which is explained just before its statement 
in [ibid.]. 

\begin{prop}\label{equiv}
Let $n$ be a positive even integer, and $\D$ be discrete series representation of $G_n$ with a linear period, then it has a Shalika functional.
\end{prop}

We also recall Proposition 3.8. of \cite{M-localBF}. As its proof contains inaccuracies, we use this 
occasion to correct them.

\begin{prop} \label{sym}
Let $m$ be a positive integer, and $\D$ be a unitary discrete series representation of $G_m$, then $\D\nu^s\times\D^\vee\nu^{-s}$ has 
a Shalika functional for all $s\in \C$.
\end{prop}
\begin{proof}
We rectify the argument in the proof of Proposition 3.8. of \cite{M-localBF} which was alluded to above. Let $\eta_s$ be the map on $G_{2m}$, defined by the relation
\[\eta_s (\mathrm{diag}(g_1 ,g_2 )uk)= \nu(g_1)^s\nu(g_2)^{-s} \]
for $g_i \in G_m$, $u$ in the standard unipotent subgroup of $G_{2m}$ of type $(m,m)$, and $k \in K_{2m}$. For
$f \in \Pi = \Pi_0$, we set 
$f_s = \eta_s f \in \Pi_s$. We denote by $L$ the linear form \[L : v\otimes v^\vee\mapsto v^\vee(v)\]
on $\Delta \otimes \Delta^\vee$.
It is claimed that the map $L(f_0):g\mapsto L(f_0)(g)$ is bounded on $G_n$ 
because it is a "coefficient" of the unitary representation $\Pi=\D\times\D^\vee$. This hints that the linear form on $\Pi$ defined by $f\mapsto L(f(I_n))$ is smooth, which is not clear. In fact, a closer look at the proof shows that it is enough to find $r\in \R$ such that $L(f_r)$ is bounded:  in this case, the integral $\l_s (f_s )$ in the
proof of Proposition 3.10. in \cite{M-localBF} will be defined by an absolutely convergent integral for
$Re(s)\geq \mathrm{max}((m-1)/2,m/2 +r)$) (instead of $Re(s)\geq m/2$ as claimed in the erroneous proof),
and this won't change the conclusion.\\
We will prove here that we can take $r=-m$. Let $U$ be a compact open subgroup of $K_{2m}$ fixing $f_{-m}$ on the left, and $k_1,\dots,k_t$ be representatives of $K_{2m}/U$. For $p=\begin{pmatrix} a & b \\ & c\end{pmatrix}$, with $a,\ c\in G_m$, and $b\in \M_m$, and $k\in K_{2m}$, there exists  $i \ (1\leq i \leq t)$ such that 
$$L(f_{-m})(pk)=L(f_{-m}(pk_i))=L(\D(a)\otimes \D^\vee(b) f_{-m}(k_i)).$$ By definition of $L$, if we write 
$f_{-m}(k_i)=\sum_j v_{i,j} \otimes v^\vee_{i,j}$, one has $$L(\D(a)\otimes \D^\vee(b) f_{-m}(k_i))
= \sum_j <\D(b^{-1}a)v_{i,j},v^\vee_{i,j}>.$$
As $\D$ is unitary, its matrix coefficients are bounded on $G_m$, hence the map 
$$(a,b)\mapsto L(\D(a)\otimes \D^\vee(b) f_{-m}(k_i))$$ is bounded on $G_m\times G_m$, and $L(f_{-m})$ is bounded on $G_n$ 
thanks to the Iwasawa decomposition.
\end{proof}

Finally, we recall Theorem 3.1. of \cite{M-localBF}.

\begin{thm}\label{linear-generic}
Let $n$ be a positive even integer, and $\pi$ be a generic representation of $GL(n, F )$.
The representation $\pi$ admits a linear  period if and only if it is obtained as a normalised
parabolic induction
\[(\D_1\times \D_1^\vee)\times \dots \times (\D_s\times \D_s^\vee)\times \D_{s+1}\times \dots \times \D_t , \ (0 \leq s \leq  t)\]
where each $\D_i$ is a discrete series, which moreover admits a linear period for $i>s$. 
\end{thm}

We are now in the position to prove Corollary \ref{cor}. Indeed, as by Proposition \ref{equiv}, a discrete series representation of 
$G_n$ has a Shalika functional if and only if it has a linear period, Theorem\ref{linear-generic} shows that the representations 
decribed in the statement of Corollary \ref{cor} are exactly the generic representations of $G_n$ with a linear period. Thanks to Proposition \ref{shalika-implies-linear}, it is thus enough to show that these representations admit a Shalika functional. But they can be written as a (commutative) product of discrete series with a Shalika functional, and of representations of the form 
$\D\times \D^\vee$, for $\D$ a discrete series, and such representations also admit a Shalika functional according to Proposition 
\ref{sym}. The statement of Corollary \ref{cor} now 
follows from Theorem \ref{main}.

\end{document}